\documentclass[12pt,a4paper]{amsart}
\usepackage[top=40mm, bottom=40mm, left=35mm, right=35mm]{geometry}
\usepackage{mathptmx}
\usepackage{mathrsfs}
\usepackage{enumerate}
\usepackage{verbatim}
\usepackage{url}
\usepackage[all]{xy}
\usepackage{color}
\usepackage{hyperref}

\theoremstyle{plain}
\newtheorem{thm}{Theorem}[section]
\newtheorem{lem}[thm]{Lemma}

\newtheorem{cor}[thm]{Corollary}

\theoremstyle{definition}
\newtheorem{de}[thm]{Definition}
\newtheorem{rem}[thm]{Remark}
\newtheorem{exam}[thm]{Example}

\newcommand{\Z}{\mathbb{Z}}
\newcommand{\R}{\mathbb{R}}
\newcommand{\N}{\mathbb{N}}

\newcommand {\ol}{\overline}
\newcommand{\ep}{\varepsilon}

\newcommand{\lra}{\longrightarrow}

\newcommand{\id}{\mathrm{id}}
\DeclareMathOperator{\diam}{diam}

\DeclareMathOperator{\Card}{Card}
\DeclareMathOperator{\Int}{Int}
\DeclareMathOperator{\orb}{orb}

\begin{document}
\title{Levels of generalized expansiveness}

\author{Jie Li}
\author{Ruifeng Zhang}
\address[J.~Li]{Wu Wen-Tsun Key Laboratory of Mathematics, USTC, Chinese Academy of Sciences and
School of Mathematics, University of Science and Technology of China,
Hefei, Anhui, 230026, P.R. China}
\email{jiel0516@mail.ustc.edu.cn}
\address[R.~Zhang] {School of Mathematics, Hefei University of Technology, Hefei, Anhui, 230009, P.R. China}
\email{rfzhang@mail.ustc.edu.cn}

\begin{abstract}
We study a class of generalized expansive dynamical systems for which at most countable orbits can be accompanied by an arbitrary given orbit. Examples of different levels of generalized expansiveness are constructed.

When the dynamical system is countable, a characterization of $n$-expansiveness is given for any natural number $n$, and as a consequence examples of dynamical systems with van der Waerden depth equal to any given countable ordinal are demonstrated, which solves open questions existing in the literature.
\end{abstract}

\date{\today}

\maketitle

\section{Introduction}

The classical term of \textit{unstable homeomorphism} ( now known as \textit{expansiveness}) first introduced by Utz in \cite{Ut50}, which is used to study the dynamical behavior saying roughly that every orbit can be accompanied by only one orbit with some certain constant.
It is clear that expansiveness implies the notion of sensitivity, which is the kernel in the definition of Devaney's chaos \cite{De89}. Hence expansiveness property involves a large class of dynamical systems exhibiting chaotic behavior, and nowadays an extensive literature has been developed on this property and its generalizations, see \cite{Ar14,AC14,APV13,Ei66,Ka93, KP99, KR69,Ma79,Mo12,MS10,Re65,Ut50,Wa00} and references therein for more knowledge.

\medskip
Among all the generalizations, the notion of \textit{$n$-expansiveness} originally introduced in \cite{Mo12} is an interesting one. Roughly speaking it loosens restriction to every orbit allows at most $n$ companion orbits with a certain constant. Note that the notion of positive $n$-expansiveness can be similarly defined when positive orbits are considered instead. Then the question that whether these generalized expansive systems can share the properties of the classical ones or not, was addressed naturally. It turns out that both positive and negative answers were provided in
\cite{Mo12}, and one particular result is that there are infinite compact metric spaces carrying positively $n$-expansive homeomorphisms for some $n\in\N$ (see \cite[Theorem 4.1]{Mo12}), which differ from the positively expansive ones.

Another natural question posed in \cite{Mo12} is that whether there are  examples of compact metric spaces admitting fixed level of positively generalized expansive homeomorphisms, i.e. positively $n$- expansive homeomorphisms that are not positively $n-1$-expansive for some integer $n\ge 2$. Note that Morales partially solved this question by showing this is true for $n=2^k\ (k\in\N)$ (see \cite[Propostion 3.4]{Mo12}). By the same spirit we can ask this classification question for all the generalized expansiveness. In \cite{APV13} the authors gave an example of a $2$-expansive homeomorphism on surface which is not expansive, and the general examples are still open. It is worth mentioning that A.~Artigue \cite{Ar14} recently introduced another variant notion of expansiveness, say \textit{$(m,l)$-expansiveness} for given integer number $m> l\ge 1$, which presents a fine division among $n$-expansiveness (see \cite[Table 1]{Ar14} for basic hierarchy), but the examples to distinguish all different hierarchies are not available too.

According to the cardinality of companion orbits, Artigue and Carrasco-Olivera in \cite{AC14} further generalize expansiveness to \textit{$\aleph_0$-expansiveness}, where $\aleph_0$ is the first countable ordinal number, and they proved that $\aleph_0$-expansive homeomorphism is equivalent to another form of generalized expansive homeomorphism in the measurable sense (see \cite[Theorem 2.1]{AC14}).

\medskip
In this paper for simplicity we introduce the notion of \textit{essential $n$-expansiveness} (resp. \textit{essential $\aleph_0$-expansiveness}) to express $n$- but not $n-1$-expansiveness (resp. countable but not finite expansiveness), and the positively essential ones are similar to introduce. One can turn to Section \ref{sect:Def} for the precise definitions and their basic properties. In the sequential Section \ref{sect:LevPosExpan} and Section \ref{sect:LevExpan}, examples of all different levels of the generalized positive expansiveness and expansiveness are given, which completely solve the question left in \cite{Mo12} (see Theorem \ref{thm:EssPosNExpan} and Theorem \ref{thm:EssNExpan}). Among other things, when the space considered has countable cardinality, it turns out that there is no compact metric space carrying positively $n$-expansive homeomorphism for any $n\in\N$ (see Theorem \ref{thm:CountPosNExpan}), which extends the classical result in some sense (compare with \cite{KR69, MS10} and \cite[Theorem 4.1]{Mo12}); moreover, under the countable assumption we can also develop a characterization of $n$-expansiveness (see Theorem \ref{thm:NExpanChar}), and note that it is a natural generalization of Kato and Park \cite[Theorem 2.2]{KP99}.

Non-wandering points play an important role in the study of dynamical systems. Parallel to this classical theory, D.~Kwietniak \textit{et al} in \cite{KLOX15} introduced the notions of \textit{multi-non-wandering point} and the corresponding \textit{van der Waerden center} and \textit{depth}. In this paper as a corollary of Theorem \ref{thm:NExpanChar} we demonstrate that the van der Waerden depth is a countable ordinal and for every countable ordinal $\alpha$ there exists a compact metric system with van der Waerden depth equal to $\alpha$ (see Corollary \ref{cor:Van}). It answers positively a conjecture left open in \cite{KLOX15}.

\section{Definitions and basic properties}\label{sect:Def}

In this paper a \textit{topological dynamical system} (abbr. t.d.s.) is a pair $(X, T)$, where $X$ is a compact metric space and $T\colon X \to X$ is a homeomorphism from $X$ into itself. When discussing the positive notions, we may loosen $T$ to a continuous surjective map. Also, throughout this paper we denote $\N$, $\Z_+$, $\Z$ and $\R$ by the sets of positive integers,
nonnegative integers, integers and real numbers, respectively.

\medskip
Let $(X,T)$ be a compact metric t.d.s. with metric $d$. Fix $\delta>0$, we put
$$
\Gamma_\delta[x,T]  =\{y\in X\colon d(T^n x, T^n y)\le\delta, \forall n\in \Z\}=\{x\}
$$
and
$$
\Phi_\delta[x,T]  =\{y\in X\colon d(T^n x, T^n y)\le\delta, \forall n\in \Z_+\}=\{x\}
$$
when $T$ is not required to be a homeomorphism. We often write $\Gamma_\delta[x]$ or $\Phi_\delta[x]$ when the acting map $T$ is clear from the context.

\begin{de}[\cite{Ut50,Ei66}]\label{De:Expa}
A homeomorphism (resp. continuous surjective map) $T$ is said to be \textit{expansive} (resp. \textit{positively expansive}) if there is an \textit{expansive constant} $\delta>0$ for $T$ such that for every $x\in X$, $\Gamma_\delta[x]=\{x\}$ (resp. $\Phi_\delta[x]=\{x\}$).
\end{de}

In \cite{Mo12} Morales first introduced the notion of \textit{$n$-expansiveness}, which is a natural generalization of the usual expansiveness.

\begin{de}\label{de:nExpa}
Let $n\in\N$. A homeomorphism (resp. continuous surjective map) $T$ is said to be \textit{$n$-expansive} (resp. \textit{$n$-positively expansive}) if there is an \textit{$n$-expansive constant} $\delta>0$ for $T$ such that for every $x\in X$, $\Gamma_\delta[x]$ (resp. $\Phi_\delta[x]$) has at most $n$ elements.
\end{de}

Clearly $1$-expansiveness is just the classical expansiveness. Now by $\aleph_0$ denote the first countable cardinality. In the same spirit Artigue and Carrasco-Olivera \cite{AC14} extend the expansiveness to the following case:

\begin{de}\label{de:CountExap}
A homeomorphism (resp. continuous surjective map) $T$ is said to be \textit{$\aleph_0$-expansive} (resp. \textit{$\aleph_0$-positively expansive}) if there is an \textit{$\aleph_0$-expansive constant} $\delta>0$ for $T$ such that for every $x\in X$, $\Gamma_\delta[x]$ (resp. $\Phi_\delta[x]$) has at most countable elements.
\end{de}

\begin{de}\label{de:essential}
we call a homeomorphism $T$ is \textit{essentially $n$/$\aleph_0$-expansive} (resp. \textit{essentially positively $n$/$\aleph_0$-expansive}) if it is $n$/$\aleph_0$-expansive (resp. positively $n$/$\aleph_0$-expansive) and for any $\delta>0$ there is at least one point $x$ such that the cardinality of $\Gamma_\delta[x]$ (resp. $\Phi_\delta[x]$) is $n/\aleph_0$ .
\end{de}

It is easy to see that a homeomorphism $T$ is essentially $n$-expansive (resp. $\aleph_0$-expansive) if and only if it is $n$- but not $n-1$-expansive (resp. countable but not finite expansive). The equivalence for the corresponding positive cases are similar to achieve.

\begin{rem}\label{rem:DefExpa}
We have the following facts:
\begin{enumerate}
  \item\label{rem:DefExpa:1} Another way to give the above concepts is to generalize the notion of generator introduced by Keynes and Robertson \cite{KR69}. That is, $T$ is $n/\aleph_0$-expansive if and only if there is a finite open cover $\alpha$ of $X$ for $T$ such that if for every bisequence $\{A_n\}_{n=-\infty}^{\infty}$ of members of $\alpha$, $\Card (\cap_{n=-\infty}^\infty T^{-n} \ol{A_n})$ is at most $n/\aleph_0$. Here $\Card(\cdot)$ means the cardinality of the set.

      From this definition we can easily see that $n/\aleph_0$-expansiveness is a topological conjugacy invariant, and it is independent of the metric as long as the metric induces the topology of $X$ (although the $n/\aleph_0$-expansive constant does change).
  \item\label{rem:DefExpa:2} It is clear that $n$-expansiveness implies $\aleph_0$-expansiveness for any $n\in\N$ and $n$-expansiveness implies $m$-expansiveness for any $m\ge n\in\N$.
  \item\label{rem:DefExpa:3} A subsystem of an (essentially) (resp. positively) $n/\aleph_0$-expansive t.d.s. is  (resp. positively) $n/\aleph_0$-expansive.
\end{enumerate}
\end{rem}

Let $(X,T)$ be a t.d.s. We say $x\in X$ is a \textit{periodic point}  if $T^n x=x$ for some $n\in\N$, and  a \textit{fixed point} if such $n=1$.  Denote by $\textrm{Per}(X,T)$ (resp. $\textrm{Fix(X,T)}$) the set of all periodic (resp. fixed) points.
Now put
$$
\alpha(x)=\{y\in X\colon \lim_{i\to -\infty} T^{n_i} x\to y\}\ \mbox{and}\ \omega(x)=\{z\in X\colon \lim_{i\to +\infty} T^{n_i} x\to z\}
$$
Call some point $x$ has \textit{converging semi-orbits under $T$} if both $\alpha(x)$ and $\omega(x)$ consist of a single point. Put $\textrm{CS}(X,T)$ as the collection of all points having converging semi-orbits under $T$.

It is well known that under the classical expansiveness assumption $\textrm{Fix(X,T)}$ is finite (see \cite[Theorem 5.26]{Wa00}) and $\textrm{Per}(X,T)$ and $\textrm{CS}(X,T)$ are countable ( for instance \cite[Theorem 3.1]{Ut50} and \cite[Theorem 1]{Re65}, respectively). Now we improve these results to the generalized expansiveness.

\begin{thm}\label{thm:fact}
We have the following generalizations:
\begin{enumerate}
  \item\label{fact:1} If $k\neq 0$ then $T$ is (resp. essentially) (resp. positively) $n/\aleph_0$-expansive if and only if so is $T^k$.
  \item\label{fact:2} If $T$ is $n$-expansive, then $\emph{Fix}(X,T)$ is finite and $\emph{Per}(X,T)$ are countable,
  \item\label{fact:3} If $T$ is $\aleph_0$-expansive, then $\emph{Fix}(X,T)$, $\emph{Per}(X,T)$ and $\emph{CS}(X,T)$ are countable.
\end{enumerate}
\end{thm}

\begin{proof}
\eqref{fact:1} We only prove that $T$ is $n/\aleph_0$-expansive if and only if so is $T^k$, and the other cases are similar. Since $T$ is continuous, there is $\ep>0$ such that whenever $d(x,y)<\ep$ then $d(T^i x, T^i y)<\delta$ for all $-k\le i\le k$ . Thus $\Gamma_\ep[x,T^k]\subset \Gamma_\delta[x,T]$ for all $x\in X$, which yields the necessity. On the other hand, we clearly have $\Gamma_\delta[x,T]\subset \Gamma_\delta[x,T^k]$, so the sufficiency holds.

\eqref{fact:2} Note that $\textrm{Per}(X,T)=\bigcup_{k\in\N} \textrm{Fix}(X,T^k)$, by \eqref{fact:1} it suffices to show $\textrm{Fix}(X,T)$ is finite whenever $T$ is $n$-expansive. Choose the $n$-expansive constant $\delta>0$ for $T$. Let $x_1,x_2,\dots,x_m\in X$ be such that $X=\bigcup_{i=1}^m B_{\delta/2}(x_i)$ due to compactness. If the contrary there are infinitely many fixed points $y_1,y_2,\dots$ in $B_{\delta/2}(x_{i_0})$ for some $1\le i_0\le m$. Put $A=\{y_1,y_2,\dots\}$. It is clear that for any $y_i\neq y_j\in A$ we have $d(T^l y_i, T^l y_j)\le \delta$ for each $l\in\Z$. However $\Card (A)$ is infinite which contradicts the definition of $\delta$. The proof ends.

\eqref{fact:3} Assume that $T$ is $\aleph_0$-expansive with $\aleph_0$-expansive constant $\delta>0$. First to claim that $\textrm{Fix}(X,T)$ and $\textrm{Per}(X,T)$ are countable. Similar as before it remains to prove $\textrm{Fix}(X,T)$ is countable whenever $T$ is $\aleph_0$-expansive. If not we apply the same manner to obtain a ball with radius $\delta/2$ containing uncountable fixed points, which yields a contradiction with the $\aleph_0$-expansive constant $\delta$. This proves the claim.

Now we show $\textrm{CS}(X,T)$ is also countable. Enumerate the countable set of fixed points as $z_1,z_2,\dots$. Consider the decomposition that
$$
\textrm{CS}(X,T)=\bigcup\nolimits_{i,j,k\in\N} \textrm{CS}(i,j,k),
$$
where
$$
\textrm{CS}(i,j,k)=\{x\in X\colon d(T^{-n} x, z_i)\le \delta/2 \mbox{ and } d(T^n x, z_j)\le \delta/2 \mbox{ for all } n\ge k\}.
$$
Clearly $\textrm{CS}(i,j,k)$ is compact. Since the countable union of countable sets is still countable, so if $\textrm{CS}(X,T)$ is uncountable, there are $i_0,j_0,k_0$ such that $\textrm{CS}(i_0,j_0,k_0)$ is uncountable.
On the other hand, by compactness we have
$$
\textrm{CS}(i_0,j_0,k_0)=\bigcup\nolimits_{s=1}^{\ t} B_{\ep_s}(x_s),
$$
where
$$
B_{\ep_s}(x_s)=\{y\in X\colon \mbox{ if } d(x_s, y)<\ep_s \mbox{ then } d(T^n x_s, T^n y)\le \delta/2 \mbox{ for all } |n|\le k_0\}.
$$
Thus there is $1\le s_0\le t$ such that $B_{\ep_{s_0}}(x_{s_0})$ contains uncountable elements. It implies that if $y_i\neq y_j\in B_{\ep_{s_0}}(x_{s_0})$ then $d(T^n y_i, T^n y_j)\le \delta$ holds for all $n\in\Z$, a contradiction with the choice of $\delta$. The proof is completed.
\end{proof}

We know that an interval or unit circle carries no expansive homeomorphisms (see for instance \cite{Re65} and \cite[Theorem 5.27]{Wa00}). Now we generalize these results to the case of $\aleph_0$-expansive homeomorphisms.

\begin{cor}\label{cor:nonexist-interval}
There is no $\aleph_0$-expansive homeomorphism of a compact interval.
\end{cor}

\begin{proof}
By Theorem \ref{thm:fact}\eqref{fact:1} we can assume that $T$ is orientation-preserving (if necessary replace $T$ by $T^2$). If $T$ is an $\aleph_0$-expansive homeomorphism acting on interval $[0,1]$, then it is easy to see that every point in $[0,1]$ has converging semi-orbit under $T$, which contradicts with Theorem \ref{thm:fact}\eqref{fact:3}.
\end{proof}

\begin{cor}\label{cor:nonexist-cicle}
There is no $\aleph_0$-expansive homeomorphism of an unit circle $S^1$.
\end{cor}

\begin{proof}
If $T$ is an $\aleph_0$-expansive homeomorphism on unit circle $S^1$, then by Corollary \ref{cor:nonexist-interval} and Remark \ref{rem:DefExpa}\eqref{rem:DefExpa:3} it has no fixed points.

By \cite[Theorem 6.18]{Wa00} there is a continuous surjection $\phi\colon S^1\lra S^1$ and a minimal rotation $S\colon S^1\lra S^1$ such that $\phi\, T=S\,\phi$, and for each $z\in S^1$ the set $\phi^{-1}(z)$ is either a point or closed interval. If each set $\phi^{-1}(z)$ is a point, then $\phi$ is a homeomorphism and $T$ is not $\aleph_0$-expansive because the minimal rotation $S$ is equicontinuous. Assume that for some $z_0$ the set $\phi^{-1}(z_0)$ is a closed interval of positive length. Since $\phi\, T=S\,\phi$, the sets $\{T^{-l}\phi^{-1}(z_0)\colon l\in\Z\}$ are mutually disjoint closed intervals. For any $\delta>0$ we can choose $N$ such that if $|l|\ge N$ the length of $T^{-l}\phi^{-1}(z_0)$ is less than $\delta$. Then by continuity of $T$ we can find a subinterval $A$ of $\phi^{-1}(z_0)$ with length less than some $\ep>0$ such that for any $a_1, a_2\in A$, $d(T^l a_1, T^l a_2)\le \delta$ for all $|l|\le N$. This implies that $\diam T^l A\le \delta$ for all $l\in \Z$. As $\Card (A)$ is uncountable, thus $\delta$ is not an $\aleph_0$-expansive constant for $T$. That is $T$ is not $\aleph_0$-expansive, a contradiction.
\end{proof}

Next we study the relationship between $\aleph_0$-expansive homeomorphisms and dimension. The definition and basic properties of dimension can be found in the book of Hurewicz and Wallman \cite{HW48}. Now we recall the notion of \textit{continuum-wise expansive homeomorphism}, which is another form of generalization first introduced by Kato \cite{Ka93}.

By a \textit{continuum} we mean a compact metric and connected non-degenerated space. A \textit{subcontinuum} is a continuum which is a subset of a space.

\begin{de}\label{de:cw-expan}
Let $(X,T)$ be a t.d.s. The homeomorphism $T$ is \textit{continuum-wise expansive} if there exists a constant $\delta>0$ such that for any nondegenerate subcontinuum $A$ of $X$, there is $n\in\Z$ such that $\diam (T^n A)>\delta$, where $\diam (A)=\sup \{d(a_1,a_2)\colon a_1,a_2\in A\}$.
\end{de}

The following lemma can be easily deduced from definitions, one can also refer to \cite{APV13,AC14}. Here we provide the details for the sake of completeness.

\begin{lem}\label{lem:CoutExp-cwExp}
If $(X,T)$ is $\aleph_0$-expansive then it is continuum-wise expansive.
\end{lem}

\begin{proof}
By definition if $X$ is $0$-dimensional then $T$ is always continuum-wise expansive. Now assume $\dim X>0$. Let $A$ be any non-degenerate subcontinuum of $X$ and $x\in A$. Then $\Card (A)$ contains uncountable elements by the non-degeneracy. Since $T$ is $\aleph_0$-expansive and assume the $\aleph_0$-expansive constant is $\delta>0$, then there exists a point $y\in A\setminus \Gamma_\delta[x]$. This implies that $d(T^n x, T^n y)>\delta$ for some $n\in\Z$, and then $\diam (T^n A)>\delta$. Thus $T$ is continuum-wise expansive with respect to $\delta$, completing the proof.
\end{proof}

An famous theorem by Ma\~n\`e \cite{Ma79} says that a compact metric space $X$ that admits an expansive homeomorphism $T$ is finite dimensional and every minimal set of $(X,T)$ is $0$-dimensional. Later Kato \cite[Theorem 5.2]{Ka93} proved that this theorem can be improved to the continuum-wise expansiveness case. By Lemma \ref{lem:CoutExp-cwExp} we immediately have the following theorem.

\begin{thm}\label{thm:dimension}
Let $(X,T)$ be an $\aleph_0$-expansive t.d.s. Then $\dim X<\infty$ and every minimal set of $T$ is $0$-dimensional.
\end{thm}

\begin{rem}\label{rem:minNExpan-minExpan}
Note that when $T$ is a $2$-expansive homeomorphism defined on a compact boundaryless surface with nonwandering set being the whole surface then $T$ is expansive (see \cite[Theorem A]{APV13}). So we may ask that if any minimal $n/\aleph_0$-expansive t.d.s. is expansive?  If this were done, then coupled with Theorem of Ma\~n\`e we immediately have $X$ is zero-dimensional.
\end{rem}

\section{Levels of positive expansiveness}\label{sect:LevPosExpan}

It is a natural question as to whether a compact metric t.d.s. can admit an essentially positively $n$-expansive homeomorphism for any $n\in\N$. In \cite[Proposition 3.4]{Mo12} Morales gave a partial answer by showing that there is a t.d.s. $(X,T)$ which is positively $2^k$-expansive but not positively $(2^k-1)$-expansive for each $k\in\N$. Motivated by this example we here display a complete solution.

\begin{thm}\label{thm:EssPosNExpan}
There exists an essentially positively $n$-expansive t.d.s. $(X,T)$ for every $n\in\N$.
\end{thm}

\begin{proof}
Note that when $n=1$ it is just the classical positive expansiveness. Now Let $n\ge 2$. To begin with we recall a concrete construction of Denjoy homeomorphism of the circle $S^1$. Let $\alpha$ be an irrational number and $T_{\alpha}\colon S^1\to S^1,\ x\mapsto x+\alpha \ (\mbox{mod}\ 1)$. It is well known that $(S^1, T_\alpha)$ is minimal, i.e. the orbit closure of each point is the whole circle. Now fix $x_0\in S^1$, ``blow up'' each point of the orbit $\{T^k_\alpha x_0\colon k\in\Z\}$ by  inserting arc $I_k= [a_k,b_k]$ (in the anticlockwise sense) such that for each $k\in \Z$,
\begin{itemize}
\item[(a)] $l(I_{k+1})=l(I_k)/2$ and $l(I_{-k})=l(I_k)$ for all $k\in\N\cup \{0\}$;
\item[(b)] $\sum_{k\in \Z} l(I_k)=1$, here $l(I)$ is the length of arc $I$.
\end{itemize}

We denote the expanded circle as $Y$ and the metric $d$ on it defined by $d(x,y)=\min \{l[x,y], l[y,x]\}$. Then there exists a monotone Denjoy homeomorphism $h\colon Y\to Y$ such that
$h(a_k)=a_{k+1},\ h(b_k)=b_{k+1},\ h(I_k)=I_{k+1}$ and $h(x)=T_\alpha(x)$ for all $x\in Y\setminus \cup_{k\in \Z} I_k$. A well known result says that every Denjoy map $h$ exhibits a unique minimal set $M_h$ which is isomorphic to a Cantor set, and in this case we have $M_h=Y\setminus \cup_{k\in\Z} (a_k, b_k)$.

To meet our needs, we modify the above construction by changing each $I_k$ to a set $A_k=\{a_k+l(I_k)\cdot i/(n-1)\colon i=0,1,\dots, n-1\}$ with cardinality $n$. 
Denote the  new space (which is a closed subset of $Y$) as $X$ and the homeomorphism is $T=h|_X$. Next we shall prove that $(X,T)$ is essentially positively $n$-expansive with respect to the metric $d|_X$. Let $0<\delta<l(I_0)/2$. We claim that $\Card (\Int(\Phi_\delta[x])\cap X)\le n-2$. Here $\Int(\cdot)$ denotes the interior operation.

To check this, we first show $\Int(\Phi_\delta[x])\cap M_h=\emptyset$ for all $x\in X$. Assume there is $y\in \Int(\Phi_\delta[x])\cap M_h$. Since $M_h$ is minimal there exists $k_s\to \infty$ with $h^{-k_s}(a_0)\to y$. By the construction we know $\{h^{-k}(I_0)\colon k\in\N\}$ is disjoint and $\sum_{k\in \Z} l(I_k)=1$, so $l(h^{-{k_s}}(I_0))\to 0$ as $s\to \infty$. It implies that $h^{-{k_s}}(I_0)\subset \Phi_\delta[x]$. Notice the fact that $h(\Phi_\delta[x])\subset \Phi_\delta[h(x)]$ we have $I_0\subset \Phi_\delta[h^{k_s}(x)]$, which is a contradiction with $0<\delta<l(I_0)/2$.

Now we check $\Card (\Int(\Phi_\delta[x])\cap (X\setminus M_h))\le n-2$. Otherwise there are at least $n-1$ distinct points $z_1,\dots,z_{n-1}\in \Int(\Phi_\delta[x])\cap (X\setminus M_h)$. Since $\Card (I_k\cap (X\setminus M_h))=n-2$ for each $k\in\Z$, without loss of generality we assume $z_1\in I_i, z_2\in I_j$ for some $i\neq j\in\Z$. As $\Phi_\delta[x]$ reduces to closed arc (possibly trivial) and $z_1,z_2\in \Int (\Phi_\delta[x])$, so the arc $[z_1, z_2]$ is contained in $\Phi_\delta[x]$. But there must have some point $z_0\in M_h\cap [z_1, z_2]$, which contradicts $\Int(\Phi_\delta[x])\cap M_h=\emptyset$ for all $x\in X$. Hence $\Card (\Int(\Phi_\delta[x])\cap (X\setminus M_h))\le n-2$ and the claims holds.

Since $\Phi_\delta[x]$ is a closed arc (possibly trivial), we have $\Card (\Phi_\delta[x]\cap M_h)\le 2$ and then $\Card (\Phi_\delta[x]\cap X)\le n$ for all $x\in X$. That is $(X,T)$ is positively $n$-expansive for the above $\delta$. Notice that by (a)(b) we also have $l(h^{k}(I_0))\to 0$ as $k\to \infty$, so for any $\delta>0$, there is $N\in\N$ such that $\diam (h^k(A_0))=\diam (A_k)=l(h^k(I_0))\le \delta$ for all $k\ge N$. This implies that $A_k\subset \Phi_\delta[a_k]\cap X$ and then $(X,T)$ is not positively $(n-1)$-expansive for $\delta$. Combining with $\Card (\Phi_\delta[x]\cap X)\le n$ we have $\Card (\Phi_\delta[a_k]\cap X)=n$, so $(X,T)$ is essentially positively $n$-expansive.
\end{proof}

A deep result in classical terms says that a compact metric space is finite once it carries a positively expansive homeomorphism, and several different proofs can be found in \cite{KR69, MS10} and the references therein. In Theorem \ref{thm:EssPosNExpan} we have shown that this is not true in the positive $n$-expansiveness case, that is there is an infinite t.d.s. $(X,T)$ carrying positively $n$-expansive homeomorphism. But if additionally $X$ is a countable space, we shall prove that the finiteness still holds.

Now we make some preparations. Call a point $x$ of $X$  an \textit{accumulation point} if $x\in \ol{X\setminus x}$. The collection of accumulation points of $X$ is said to be the \textit{derived set} of $X$, write as $X^d$. The derived set of $X$ of order $\alpha$ is recursively defined by the conditions: $X^{(1)}=X^d, \ X^{(\alpha+1)}=(X^{(\alpha)})^d$ and $X^{(\lambda)}=\bigcap_{\alpha<\lambda} X^{(\alpha)}$ if $\lambda$ is a limit ordinal. We denote the \textit{derived degree} of $X$ by $d(X)$ and $d(X)=\alpha$ if $X^{(\alpha)}\neq\emptyset$ and $X^{(\alpha+1)}=\emptyset$.

Note that a compact metric space $X$ is a countable set if and only if $d(X)$ exists and it is a countable ordinal number. Also, it is clear that if $d(X)=\alpha$ then $X^{(\alpha)}$ is a finite set. We leave readers to the book by Kuratowski \cite[p. 261]{Ku68} for more details.

\begin{thm}\label{thm:CountPosNExpan}
Let $(X,T)$ be a t.d.s. with countable infinite cardinality. Then $X$ carries no positively $n$-expansive homeomorphism $T$ for any $n\in\N$.
\end{thm}

\begin{proof}
Since $X$ is countable, we can choose arbitrarily small radii such that the balls below are open and closed. Let $d(X)=\alpha$ with $\alpha$ a countable ordinal and denote $X^{(\alpha)}=\{x_1,\dots,x_n\}$. By Theorem \ref{thm:fact}\eqref{fact:1} each $x_i$ can be assumed to be fixed point. Let $\ep>0$ be such that $B_\ep(x_1),\dots, B_\ep(x_n)$ are pairwise disjoint. Choose $0<\delta<\ep/2$ such that
\begin{equation*}
  B_\delta(x_i)\subset B_\ep(x_i)\ \mbox{and} \ T B_\delta(x_i)\subset B_\ep(x_i). \eqno{(*)}
\end{equation*}
Note that each $B_\delta(x_i)$ contains countable infinite elements. Now we consider the following property
\begin{align*}
  P(\alpha) \colon & \mbox{ if } d(X)=\alpha, \ \mbox{there is} \ \ 1\le i_\alpha\le n \ \mbox{such that}  \\
   & Y_{i_\alpha} =\{x\in X\colon \ol{\textrm{orb}^+(x, T)}\subset B_\delta(x_{i_\alpha})\} \ \mbox{is countable infinite.}
\end{align*}
Here $\ol{\textrm{orb}^+(x, T)}=\ol{\{x, T x, \dots\}}$ means the positive orbit closure of $x$ under $T$. Since for any $x,y\in Y_{i_\alpha}$ we have for each $l\in\Z_+$,
$$
d(T^l x,T^l y)<d(T^l x,x_{i_\alpha})+d(x_{i_\alpha},T^l y)<\ep,
$$
so if $P(\alpha)$ holds then $(X,T)$ is not positively $n$-expansive for any $n\in\N$. Next we would follow this idea to check the validity of $P(\alpha)$ by transfinite discussion.

\smallskip
Assume that $\alpha$ is not a limit ordinal, that is $\alpha=\beta+1$ for some ordinal number $\beta$. Then two cases are involved: (i) $X^{(\beta)}$ is pointwise periodic and (ii) there is a non-periodic point $y\in X^{(\beta)}$.

For the case (i), we claim that
$$
Z_\beta=\{x\in X\colon x\in X^{(\beta)} \ \mbox{and}\ \ol{\textrm{orb}^+(x, T)}\not\subset \cup_{i=1}^n B_\delta(x_i)\}
$$
 is finite. In fact if the contrary  then by $T(X^{(\beta)})\subset X^{(\beta)}$ we have a limit point $x_0\in X\setminus (\cup_{i=1}^n B_\delta(x_i))\cap X^{(\alpha)}=\emptyset$, a contradiction. Hence $B_\delta(x_i)\cap Z_\beta$ is finite for each $1\le i\le n$. On the other hand, by condition $(*)$ we get  $Y_i\supset X^{(\beta)}\cap B_\delta(x_i)\setminus Z_\beta$, and note that $X^{(\beta)}\cap B_\delta(x_i)$ is countable infinite  then so is $Y_i$ for every $1\le i\le n$. This case ends.

Now consider the case (ii). First we note that $X\setminus (\cup_{i=1}^n B_\delta(x_i))\cap X^{(\beta)}$ is finite. Then combine with the fact $T(X^{(\beta)})\subset X^{(\beta)}$ and the condition that $y\in X^{(\beta)}$ is not periodic, we can declare that there exist $m\in\N$ and $1\le i_\alpha\le n$ such that $T^m y\in Y_{i_\alpha}$, and then $\ol{\textrm{orb}^+(T^m y, T)}\subset Y_{i_\alpha}$. As the cardinality of $\ol{\textrm{orb}^+(T^m y, T)}$ is countable infinite, we have $P(\alpha)$ is true.

\smallskip
Assume that $\alpha$ is a limit ordinal number. Since $\bigcap_{\gamma<\alpha} X^{(\gamma)}=X^{(\alpha)}$, then there exists an ordinal $\gamma_{_0}<\alpha$ such that $X^{(\gamma)}\subset \cup_{i=1}^n B_\delta(x_i)$ for all $\gamma>\gamma_{_0}$. Note that $T(X^{(\gamma)})\subset X^{(\gamma)}$ for any $\gamma<\alpha$.
By the continuity of $T$, there is $\sigma>0$ such that if $z\in X^{(\gamma)}\cap B_\sigma(x_i)$ then $T z\in X^{(\gamma)}\cap B_{\delta}(x_i)$ for all $\gamma>\gamma_{_0}$ and all $i=1,\dots, n$. For this $\sigma>0$, we choose suitable $\gamma_{_0}<\gamma_{_1}<\alpha$ such that $X^{(\gamma_{_1})}\subset \cup_{i=1}^n B_\delta(x_i)$. Hence for each $1\le i\le n$ we have
$$T(X^{(\gamma_{_1})}\cap B_{\delta}(x_i))=T(X^{(\gamma_{_1})}\cap B_{\sigma}(x_i))\subset X^{(\gamma_{_1})}\cap B_{\delta}(x_i).$$
And inductively we see
$$T^k(X^{(\gamma_{_1})}\cap B_{\delta}(x_i))\subset X^{(\gamma_{_1})}\cap B_{\delta}(x_i) \ \ \mbox{for any} \ \ k\in\Z.$$
That is $X^{(\gamma_{_1})}\cap B_{\delta}(x_i)\subset Y_i$ and then $P(\alpha)$ is true. We are done.
\end{proof}

\begin{rem}
We point out that Theorem \ref{thm:CountPosNExpan} actually presents a speical class of essentially positively $\aleph_0$-expansive systems, and in this countable case, essentially positive $\aleph_0$-expansiveness is equivalent to positive non-$n$-expansiveness for any $n\in\N$. But in general they are not the same, for example see Corollary \ref{cor:nonexist-cicle}.
\end{rem}

\section{Levels of expansiveness}\label{sect:LevExpan}
Parallel to the previous section, we naturally ask if there exist examples to distinguish all the levels of expansiveness. It is not hard to check that the example of Theorem \ref{thm:EssPosNExpan} is essentially positively $n$-expansive but fails to be essentially $n$-expansive for any $n\ge 2$. In fact it is expansive (to check $\Card (\Gamma_\delta [x]\cap X)=1$). In \cite{Ar14} Artigue studied $(m,l)$-expansiveness ($m>l\ge 1$) and showed that there are $(4,2)$-expansive homeomorphisms that are not $(3,2)$-expansive (\cite[Proposition 4.3]{Ar14}), but we know that $(4,2)$-expansiveness implies $(2,1)$-expansiveness which is equivalent to expansiveness (\cite[Proposition 1.4]{Ar14}). Besides, in \cite{APV13} the authors worked out an example of a $2$-expansive but not expansive homeomorphism on surface, but the general case is still open. In this section we shall present a complete answer.

\subsection{Essential $n$-expansiveness}
Note that the next characterization is mainly inspired by Kato and Park \cite{KP99}, and for convenience we follow their notations and terms. Let $S=\{s_i\colon i\in\Z\}\cup \{s_\infty\}$ with $\lim_{i\to +\infty}s_i=\lim_{i\to -\infty}s_i=s_\infty$ and $s_i, s_\infty$ be points in the plane with $s_i\neq s_j\ (i\neq j),\ s_i\neq s_\infty$. We define a homeomorphism $g\colon S\to S$ by $g(s_i)=s_{i+1}$ and $g(s_\infty)=s_\infty$.

Let $r\in\N$ and $U(s_{-r}),\dots,U(s_r),U(s_\infty)$ be closed neighborhoods of  $s_{-r},\dots,s_r,s_\infty$ in the plane respectively. Require $U(s_i)\cap U(s_i)=\emptyset,\ i\neq j$ and $S\subset \cup_{i=-r}^r U(s_i)\cup U(s_\infty)$. Set $V(=V_r)=\cup_{i=-r}^r U(s_i)\cup U(s_\infty)$ and call $V$ as a \textit{neighborhood system} of $S$. Assume a sequence $\{t_i\colon i\in\Z\}\subset V$ with $\lim_{i\to +\infty}t_i=\lim_{i\to -\infty}t_i=s_\infty$. Fix $d\in\Z_+$. We say $\{t_i\colon i\in\Z\}$ \textit{winds d-times around $S$} (with respect to $V$), if there is $k\in\Z_+$ satisfying that
\begin{enumerate}
  \item[(a)] $k>2r$,
  \item[(b)] $t_{jk+i}\in U(s_{-r+i})$ for all $0\le i\le 2r$, $1\le j\le d$, and
  \item[(c)] $t_m\in U(s_\infty)$ for other $t_m$.
\end{enumerate}
For brevity we denote $w_S(\{t_i\colon i\in\Z\}; V)=d$. Here we remark that the conditions (a)-(c) have a small difference with the original ones in \cite{KP99}, but they are still available for some well chosen sequence $\{t_i\colon i\in\Z\}$. The intention of this change will be revealed when consider the van der Waerden depth (see \cite{KLOX15}) later.

Consider a countable t.d.s. $(X, T)$ with $S\subset X\subset V$ and $T\colon X\to X$ an extension of $g$. Let $d\in\Z_+$. We say a point $x\in X$ has \textit{winding number of d} if $\textrm{orb}(x,T)$ winds $d$-times around $S$ (with respect to $V$), and denote $w_S(x; T, V)=d$. It is clear that $w_S(s_i; T, V)=1$ for $s_i \in S\ (i\neq \infty)$ and $w_S(s_\infty; T, V)=0$. Let $Y$ be a subset of $X$, we write $w_S(Y;T,V)=\{w_S(x; T, V)\colon x\in Y\}$ and call $w_S(Y;T,V)$ is the \textit{winding set} of $Y$ around $S$.

\begin{thm}\label{thm:EssNExpan}
There exists an essentially $n$-expansive t.d.s. $(X,T)$ for every $n\in\N$.
\end{thm}

\begin{proof}
Note that the above t.d.s. $(S,g)$ is expansive. Now let $n\ge 2$. Assume $M$ is any infinite subset of prime numbers and $V$ is any neighborhood system of $S$. We shall construct a countable t.d.s. $(X,T)$ satisfying the following conditions:
\begin{enumerate}
    \item[(a)] $S\subset X\subset V$, $d(X)=2$ and $X^{(2)}=\{s_\infty\}$;
    \item[(b)] $T\colon X\to X$ is a homeomorphism and $T|_S=g\colon S\to S$;
    \item[(c)] for any $x\in X\setminus \{s_\infty\}$, $w_S(x; T,V)\neq 0$ and $w_S(X\setminus S; T,V)\subset M$;
    \item[(d)] for each pair $x_1,x_2\in X\setminus \{s_\infty\}$, we have either $\orb (x_1,T)=\orb (x_2,T)$ or $\orb (x_1,T)\cap \orb (x_2,T)=\emptyset$. Moreover,
    \item[ ] \begin{itemize}
         \item[(d1)] if $\orb (x_1,T)=\orb (x_2,T)$ then  $w_S(x_1; T, V)= w_S(x_2;T, V)$, and
         \item[(d2)] if $\orb (x_1,T)\cap \orb (x_2,T)=\emptyset$, then either $w_S(x_1; T,  V)\neq w_S(x_2; T,V)$, or there are $n$, and only $n$, pairwise disjoint orbits, say $\orb (x_i,T)$, $i=1,\dots,n$, such that $w_S(x_i; T,  V)= w_S(x_j; T,V)$ for every $1\le i,\ j\le n$;
      \end{itemize}
    \item[(e)] for any $\delta>0$, there exist $y_1,\dots, y_n\in X$ such that $\Gamma_\delta[y_i]=\{y_1,\dots,y_n\}$ for any $1\le j\le n$.
\end{enumerate}
It is easy to see that provided with the conditions above, $(X,T)$ is essentially $n$-expansive for any $n\ge 2$.

Now we give the construction. Denote $M=\{p_1,p_2,\dots\}$ with $p_i\neq p_j$ for any $i\neq j$. Choose a descending neighborhood systems family $V=V_1,\ V_2\dots$ of $S$ with $\bigcap_{i=1}^{+\infty} V_i=S$. For every $i\in \N$, add $n-1$ suitable  decreasing neighborhood systems $W_1^i,\dots,W_{n-1}^i$ such that
$V_{i}\supset W_1^i\supset \dots\supset W_{n-1}^i\supset V_{i+1}$. Put $V_i=W_0^i,\ V_{i+1}=W_n^i$, then for each $i\in\N$ and $1\le m\le n$ we pick carefully a sequence $\{x_{m,j}^i\colon j\in\Z\}$ consisting of different points in $W_{m-1}^i\setminus W_m^i$, such that $\lim_{j\to +\infty}x_{m,j}^i=\lim_{j\to -\infty}x_{m,j}^i=s_\infty$ and  $w_S(\{x_{m,j}^i\colon j\in\Z\}; V_i)=p_i$. Write $V_i=\bigcup_{j=-r_i}^{r_i} U_i(s_j)\bigcup U_i(s_\infty) $ and further we assume $\{x_{m,j}^i\colon m=1,\dots,n\}\subset U_i(s_{i_j})$ for each $i\in\N$ and $j\in\Z$.
Renumber all the sequences $\{x_{m,j}^i\colon j\in\Z\}$ well-ordered as $Y_1,Y_2,\dots$. We can see that each $Y_i$ is homeomorphic to $S$, $Y_i\cap S=Y_i\cap Y_j=\{s_\infty\}$ for $i\neq j$ and
$\lim_{i\to +\infty}H_d(Y_i,S)=0$. Take $X=\bigcup_{i\in\N} Y_i\cup S$ and define $T\colon X\to X$ by $T|_S=g\colon S\to S$, $T(x_{m,j}^i)=x_{m,j+1}^i$. Then $T$ is a homeomorphism and $d(X)=2$, $X^{(2)}=\{s_\infty\}$. It is not hard to see that $(X,T)$ satisfies all the conditions and so it is the desired system.
\end{proof}

In \cite{KP99} Kato and Park showed that $X$ admits an expansive homeomorphism if and only if its derived degree is not a limit ordinal number. We can improve this deep result to  $n$-expansiveness case.

The following Lemma can be found in \cite[Proposition 2.1]{KP99}.

\begin{lem}\label{lem:Count}
Let $X,\ Y$ be two countable metric spaces with $d(X)=d(Y)=\alpha$. If  $X^{(\alpha)}$ and $Y^{(\alpha)}$ are homeomorphic, then so are $X$ and $Y$.
\end{lem}

\begin{thm}\label{thm:NExpanChar}
Let $(X,T)$ be a t.d.s. with $d(X)=\alpha\ge 2$ and $n\in\N$. Then $X$ admits an essentially $n$-expansive homeomorphism if and only if $\alpha$ is not a limit ordinal number.
\end{thm}

\begin{proof}
Note that at the end part of the proof in Theorem \ref{thm:CountPosNExpan}, we in fact showed that if $\alpha$ is a limit ordinal number then $X$ admits no $n$-expansive homeomorphism for any $n\in\N$. So it remains to show the sufficiency.

Assume $\alpha\ge 2$ is not a limit ordinal. As $d(X)=\alpha$ we have $X^{(\alpha)}=\{a_1,\dots,a_n\}$. Choose suitable closed subsets $Z_1,\dots,Z_n$ of $X$ such that $X=\bigcup_{i=1}^n Z_i$, $Z_i\cap Z_j=\emptyset$ and $d(Z_i)=\{a_i\}$. Next we shall inductively construct an essentially $n$-expansive t.d.s. $(X_\alpha,T_\alpha)$ in the plane with $d(X_\alpha)=\alpha$ and $X^{(\alpha)}=\{s_\infty\}$. If this were done, then by Lemma \ref{lem:Count} and Remark \ref{rem:DefExpa}\eqref{rem:DefExpa:1} there exists an essentially $n$-expansive homeomorphism on every $Z_i$. It implies that we can find an essentially $n$-expansive homeomorphism on $X$ and the proof ends.

Now we construct the desired t.d.s. $(X_\alpha,T_\alpha)$, and similar as Theorem \ref{thm:EssNExpan} it suffices to fulfil the below property $P(\alpha)$:
\begin{enumerate}[ ]
\item Let $M$ be an infinite subset of prime numbers and $V$ be a neighborhood system of $S$ in the plane. Then we can construct a countable t.d.s. $(X_\alpha,T_\alpha)$ such that:
\item[ ]
\begin{enumerate}
    \item[(a)] $S\subset X_\alpha\subset V$, $d(X)=\alpha$ and $X_\alpha^{(\alpha)}=\{s_\infty\}$;
    \item[(b)] $T_\alpha\colon X_\alpha\to X_\alpha$ is a homeomorphism and $T_\alpha|_S=g\colon S\to S$;
    \item[(c)] for any $x\in X_\alpha\setminus \{s_\infty\}$, $w_S(x; T_\alpha,V)\neq 0$ and $w_S(X_\alpha\setminus S;
        T_\alpha,V)\subset M'$, where $M'=\{q_1\cdots q_t\colon q_i\in M, i=1,\dots,t, t\in\N\}$;
    \item[(d)] for each pair $x_1,x_2\in X_\alpha\setminus \{s_\infty\}$, we have either $\orb (x_1,T_\alpha)=\orb (x_2,T_\alpha)$ or $\orb (x_1,T_\alpha)\cap \orb (x_2,T_\alpha)=\emptyset$. Moreover,
    \item[ ] \begin{itemize}
        \item[(d1)] if $\orb (x_1,T_\alpha)=\orb (x_2,T_\alpha)$ then  $w_S(x_1; T_\alpha, V)= w_S(x_2;T_\alpha, V)$, and
         \item[(d2)] if $\orb (x_1,T_\alpha)\cap \orb (x_2,T_\alpha)=\emptyset$, either $w_S(x_1; T_\alpha,  V)\neq w_S(x_2; T_\alpha,V)$, or there are $n$, and only $n$, pairwise disjoint orbits, say $\orb (x_i,T_\alpha)$, $i=1,\dots,n$, such that $w_S(x_i; T_\alpha,  V)= w_S(x_j; T_\alpha,V)$ for every $1\le i, \ j\le n$;
      \end{itemize}
    \item[(e)] for any $\delta>0$, there exist $y_1,\dots, y_n\in X$ such that $\Gamma_\delta[y_j]=\{y_1,\dots,y_n\}$ for any $1\le j\le n$.
\end{enumerate}
\end{enumerate}

It turns out that $P(2)$ is just the Theorem \ref{thm:EssNExpan}.
Now assume that $\alpha=\beta+1$. We consider two cases: (i) $\beta$ is a non-limit ordinal number; (ii) $\beta$ is a limit ordinal number.

For the case (i), by induction we assume $P(\beta)$ holds and the aim is to prove $P(\alpha)$ holds too. Let $M$ and $V$ be given as described in $P(\alpha)$. Divide $M=M_0\cup M_1\cup\cdots$ with $M_i\cap M_j=\emptyset$, $i\neq j$ and $\Card (M_i)=\infty$. Also, we list $M_0$ as $\{p_1,p_2,\dots\}$. Next select a descending family of neighborhood systems of $S$, say $V=V_1\supset V_2\supset\cdots$, such that $\bigcap_{i=1}^{+\infty}V_i=S$.

For each $i\in \N$, by induction we can construct a countable t.d.s. $(X_\beta^i, T_\beta^i)$ satisfying that:
\begin{enumerate}
    \item[(a)] $S\subset X_\beta^i \subset V'_i\subset V_i$, $d(X_\beta^i)=\beta$ and $(X_\beta^i)^{(\beta)}=\{s_\infty\}$, where the well-chosen $V'_i$ is a sufficiently small neighborhood system and can be asked to meet the constraints related to $V_i$ and $p_i$ later;
    \item[(b)] $T_\beta^i\colon X_\beta^i\to X_\beta^i$ is a homeomorphism and $T_\beta^i|_S=g\colon S\to S$;
    \item[(c)] for any $x^i\in X_\beta^i\setminus \{s_\infty\}$, $w_S(x^i; T_\beta^i,V'_i)\neq 0$ and $w_S(X_\beta^i\setminus S;
        T_\beta^i,V'_i)\subset M_i^*$;
    \item[(d)] for each pair $x_1^i,x_2^i\in X_\beta^i\setminus \{s_\infty\}$, we have either $\orb (x_1^i,T_\beta^i)=\orb (x_2^i,T_\beta^i)$ or $\orb (x_1^i,T_\beta^i)\cap \orb (x_2^i,T_\beta^i)=\emptyset$. Moreover,
    \item[ ] \begin{itemize}
        \item[(d1)] if $\orb (x_1^i,T_\beta^i)=\orb (x_2^i,T_\beta^i)$ then  $w_S(x_1^i; T_\beta^i, V'_i)= w_S(x_2^i;T_\beta^i, V'_i)$, and
         \item[(d2)] $\orb (x_1^i,T_\beta^i)\cap \orb (x_2^i,T_\beta^i)=\emptyset$ then either $w_S(x_1^i; T_\beta^i,  V'_i)\neq w_S(x_2^i; T_\beta^i,V'_i)$, or there are $n$, and only $n$, pairwise disjoint orbits, say $\orb (x_j^i,T_\beta^i)$, $j=1,\dots,n$, such that $w_S(x_j^i; T_\beta^i,  V_i')= w_S(x_k^i; T_\beta^i,  V_i')$ for every $1\le j,\ k\le n$;
      \end{itemize}
    \item[(e)] for any $\delta>0$, there exist $y_1^i,\dots, y_n^i\in X_\beta^i$ such that $\Gamma_\delta[y_j^i]=\{y_1^i,\dots,y_n^i\}$ for any $1\le j\le n$.
\end{enumerate}

Note that the sequence $\{s_j\colon j\in\Z\}\subset X_{\beta}^i$ and $w_S(s_j;T_{\beta}^i, V'_i)=1$ for each $i\in\N$ and $j\neq\infty$. To ensure the $n$-expansiveness, for each $i\in \N$ by well choosing the above $V'_i$ we can find a continuous embedding $\psi^i\colon X_\beta^i\to V_i$, such that $\psi^i(s_\infty)=s_\infty$, $w_S(\{\psi^i(s_j)\colon j\in
\Z\}; V_i)=p_i$ and $w_S(\{\psi^i(x^i_j)\colon j\in\Z\}; V_i)\in p_i\cdot M_i^*$ for any point $x^i_j\in X_\beta^i\setminus S$. Set $Z_\beta^i=\psi^i(X_\beta^i)$ and we may further assume $Z_\beta^i\cap Z_\beta^j=Z_\beta^i\cap S=\{s_\infty\}$ for any $i\neq j\in \Z$. Take $X_\alpha=\bigcup_{i=1}^{+\infty}Z_\beta^i\cup S$ and define $T_\alpha\colon X_\alpha\to X_\alpha$ by $T_\alpha|_S=g$ and $T_\alpha|_{Z_\beta^i}=\psi^i\circ T_\beta^i \circ (\psi^i)^{-1}$. It is easy to check that $(X_\alpha,T_\alpha)$ meets the conditions of $P(\alpha)$ and this case ends.

As for the case (ii), since $P(\beta)$ is not true as the necessity showed, then we assume $P(\gamma)$ is true for any $\gamma<\beta$. We attend to show $P(\alpha)$ is true too. The proof of this case is similar to the above one, and for the sake of completeness we present the details. Divide $M=M_0\cup M_1\cup\cdots$ with $M_i\cap M_j=\emptyset$, $i\neq j$ and $\Card (M_i)=\infty$, where $M_0=\{p_1,p_2,\dots\}$. Pick a well-ordered sequence of non-limit ordinals $\gamma_{_1}<\gamma_{_2}<\dots$ such that $\lim_{i\to +\infty}\gamma_{_i}=\beta$. Also, pick a descending family of neighborhood systems of $S$, say $V=V_1\supset V_2\supset\cdots $, such that $\bigcap_{i=1}^{+\infty}V_i=S$.

For each $i\in \N$, by induction we can construct a countable t.d.s. $(X_{\gamma_{_i}}, T_{\gamma_{_i}})$ satisfying that:
\begin{enumerate}
    \item[(a)] $S\subset X_{\gamma_{_i}} \subset V'_i \subset V_i$, $d(X_{\gamma_{_i}})=\gamma_{_i}$ and $(X_{\gamma_{_i}})^{({\gamma_{_i}})}=\{s_\infty\}$ ($V'_i$ is similar as above);
    \item[(b)] $T_{\gamma_{_i}}\colon X_{\gamma_{_i}}\to X_{\gamma_{_i}}$ is a homeomorphism and $T_{\gamma_{_i}}|_S=g\colon S\to S$;
    \item[(c)] for any $x^i\in X_{\gamma_{_i}}\setminus \{s_\infty\}$, $w_S(x^i; T_{\gamma_{_i}},V'_i)\neq 0$ and $w_S(X_{\gamma_{_i}}\setminus S;
        T_{\gamma_{_i}},V'_i)\subset M_i'$;
    \item[(d)] for each pair $x_1^i,x_2^i\in X_{\gamma_{_i}}\setminus \{s_\infty\}$, we have either $\orb (x_1^i,T_{\gamma_{_i}})=\orb (x_2^i,T_{\gamma_{_i}})$ or $\orb (x_1^i,T_{\gamma_{_i}})\cap \orb (x_2^i,T_{\gamma_{_i}})=\emptyset$. Moreover,
    \item[ ] \begin{itemize}
         \item[(d1)] if $\orb (x_1^i,T_{\gamma_{_i}})=\orb (x_2^i,T_{\gamma_{_i}})$ then  $w_S(x_1^i; T_{\gamma_{_i}}, V'_i)= w_S(x_2^i;T_{\gamma_{_i}}, V'_i)$, and
         \item[(d2)] $\orb (x_1^i,T_{\gamma_{_i}})\cap \orb (x_2^i,T_{\gamma_{_i}})=\emptyset$ then either $w_S(x_1^i; T _{\gamma_{_i}}, V'_i)\neq w_S(x_2^i; T_{\gamma_{_i}},V'_i)$, or there are $n$, and only $n$, pairwise disjoint orbits, say $\orb (x_j^i,T_{\gamma_{_i}})$, $j=1,\dots,n$, such that $w_S(x_j^i; T_{\gamma_{_i}},  V_i')= w_S(x_k^i; T_{\gamma_{_i}},  V_i')$ for every $1\le j,\ k\le n$;
      \end{itemize}
    \item[(e)] for any $\delta>0$, there exist $y_1^i,\dots, y_n^i\in X_{\gamma_{_i}}$ such that $\Gamma_\delta[y_j^i]=\{y_1^i,\dots,y_n^i\}$ for any $1\le j\le n$.
\end{enumerate}

For each $i\in \N$, we choose a continuous embedding $\phi^i\colon X_{\gamma_{_i}}\to V_i$, such that $\phi^i(s_\infty)=s_\infty$, $w_S(\{\phi^i(s_j)\colon j\in
\Z\}; V_i)=p_i$ and $w_S(\{\phi^i(x^i_j)\colon j\in\Z\}; V_i)\in p_i\cdot M_i^*$ for any point $x^i_j\in X_{\gamma_{_i}}\setminus S$. Set $Z_{\gamma_{_i}}=\phi^i(X_{\gamma_{_i}})$ and further assume $Z_{\gamma_{_i}}\cap Z_{\gamma_{_i}}=Z_{\gamma_{_i}}\cap S=\{s_\infty\}$ for any $i\neq j\in \Z$. Take $X_\alpha=\bigcup_{i=1}^{+\infty}Z_{\gamma_{_i}}\cup S$ and define $T_\alpha\colon X_\alpha\to X_\alpha$ by $T_\alpha|_S=g$ and $T_\alpha|_{Z_{\gamma_{_i}}}=\phi^i\circ T_{\gamma_{_i}} \circ (\phi^i)^{-1}$. It is easy to check that $(X_\alpha,T_\alpha)$ fulfils the conditions of $P(\alpha)$ and we complete the whole proof.
\end{proof}

Denote $\mathcal{H}(X)$ all the homeomorphism on $X$. It is a metrizable  space with metric defined by
$$
D(T_1,T_2)=\max \{d(T_1 (x), T_2 (x))\colon x\in X\}
$$
for any $T_1,T_2\in \mathcal{H}(X)$. For every $n\in
\N$ put
$$
\mathcal{E}_n(X)=\{T\in \mathcal{H}(X)\colon T \ \mbox{ is an } n \mbox{-expanive homeomorphism}\}.
$$
Then we have a direct corollary of Theorem \ref{thm:NExpanChar}.
\begin{cor}
Let $(X,T)$ be a countable t.d.s. with $d(X)=\alpha$ and $n\in\N$. If $\alpha$ is a non-limit ordinal number and $\alpha>\aleph_0$, then $\ol{\mathcal{E}_n(X)} \neq\mathcal{H}(X)$.
\end{cor}

\begin{proof}
Choose $x\in X^{(\aleph_0)}\setminus X^{(\aleph_0+1)}$ then there is an open and closed neighborhood $U$ of $x$ with derived degree $d(U)=\aleph_0$. Let $n\in\N$ and $0<\ep<d(X\setminus U,U)=\min \{d(y,z)\colon y\in X\setminus U, z\in U)\}$. If there is a homeomorphism $T$ such that $D(T,\id)<\ep$, then $T(U)=U$. That is $U$ is closed and $T$-invariant, then $(U, T|_U)$ forms a subsystem and $d(U)=\aleph_0$. By Theorem \ref{thm:NExpanChar} we know that $T|_U$ can not be $n$-expansive for any $n\in\N$, and then neither is $T$ by Remark \ref{rem:DefExpa}\eqref{rem:DefExpa:3}.
\end{proof}

\subsection{Van der Waerden depth}
Another consequence of Theorem \ref{thm:NExpanChar} is related to the notion of Van der Waerden depth, and to describe this we first recall some notions.

Let $(X,T)$ be a t.d.s. and $x\in X$. We say $x$ is a \textit{non-wandering point} if for each neighborhood $U$ of $x$ there is $n\in\N$ with $U\cap T^{-n}U\neq\emptyset$. Write $\Omega(X,T)$ the collection of all non-wandering points. It is well known that $(\Omega(X,T),T)$ forms a subsystem of $(X,T)$, and we can consider $\Omega(\Omega(X,T),T)$ in a natural way. Note that there exists a system $(X,T)$ such that $\Omega(\Omega(X,T),T)\neq \Omega(X,T)$ (see \cite{KP99} for example). More generally, by induction we set $\Omega_0(X,T)=X$, $ \Omega_1(X,T)=\Omega(\Omega_0(X,T),T)$, $\Omega_{\alpha+1}(X,T)=\Omega(\Omega_\alpha(X,T),T)$ and
$\Omega_{\lambda}(X,T)=\bigcap_{\alpha<\lambda} \Omega_\alpha(X,T)$ if $\lambda$ is a limit ordinal number.
A well known conclusion says that descending family of closed subsets in compact metric space is always at most countable, so there exists a countable ordinal $\alpha$
satisfying $\Omega_\alpha(X,T)=\Omega_{\alpha+1}(X,T)$. Denote the \textit{depth} of $(X,T)$ as
$$\textrm{depth}(X,T)=\min \{\alpha\colon \Omega_\alpha(X,T)=\Omega_{\alpha+1}(X,T)\}$$
and call $\Omega_\alpha(X,T)$ as the \textit{Birkhoff center} of $(X,T)$. It is well known that there exists a t.d.s. $(X,T)$ with $\textrm{depth}(X,T)=\alpha$ when $\alpha$ is a countable ordinal (see for instance \cite{Ne78} and \cite[Corollary 2.7]{KP99}).

Similar as above, the authors in \cite{KLOX15} introduced multi-non-wandering points and the van der Waerden center.

We say $x$ is \textit{multi-non-wandering} if for each neighborhood $U$ of $x$ and each $d\in\N$, there
is $k\in\N$ such that
$$U\cap T^{-k}U\cap T^{-2k}U\cap\dots \cap T^{-dk}U\neq\emptyset.$$
Denote by $\Omega^{(\infty)}(X,T)$ the set of all multi-non-wandering points. It is easy to see that $(\Omega^{(\infty)}(X,T),T)$ can form a subsystem of $(\Omega(X,T),T)$. Note that there also exists a system $(X,T)$ such that $\Omega^{(\infty)}(\Omega^{(\infty)}(X,T),T)\neq \Omega^{(\infty)}(X,T)$ (see \cite[Example 6.7]{KLOX15}). Inductively we put $\Omega_0^{(\infty)}(X,T)=X$, $ \Omega_1^{(\infty)}(X,T)=\Omega^{(\infty)}(\Omega_0^{(\infty)}(X,T),T)$, $\Omega_{\alpha+1}^{(\infty)}(X,T)=\Omega^{(\infty)}(\Omega_\alpha^{(\infty)}(X,T),T)$ and $\Omega_{\lambda}^{(\infty)}(X,T)=\bigcap_{\alpha<\lambda} \Omega_\alpha^{(\infty)}(X,T)$ when $\lambda$ is a limit ordinal number. We call $\Omega_{\lambda}^{(\infty)}(X,T)$ the \textit{van der Waerden center} of $(X,T)$ if $\Omega_{\alpha}^{(\infty)}(X,T)=\Omega_{\alpha+1}^{(\infty)}(X,T)$ and denote by $\textrm{depth}(X,T)$ the \textit{van der Waerden depth} of $(X,T)$ defined as
$$\textrm{depth}^{(\infty)}(X,T)=\min \{\alpha\colon \Omega_\alpha^{(\infty)}(X,T)=\Omega_{\alpha+1}^{(\infty)}(X,T)\}.$$
The same reason as above we know $\textrm{depth}^{(\infty)}(X,T)$ is a countable ordinal number.

Denote $\aleph_1$ the first uncountable ordinal number.

\begin{thm}\label{thm:Van}
Assume $\alpha<\aleph_1$ is a non-limit ordinal number and $n\ge 2$. Then there is a countable t.d.s. $(X_\alpha,T_\alpha)$ such that
\begin{enumerate}
\item $d(X_\alpha)=\emph{depth}^{(\infty)}(X_\alpha,T_\alpha)=\emph{depth}(X_\alpha,T_\alpha)=\alpha$, and
\item $(X_\alpha,T_\alpha)$ is essentially $n$-expansive.
\end{enumerate}
\end{thm}

\begin{proof}
We point out that the construction of Theorem \ref{thm:NExpanChar} is just what we desired. Here we only check the case that $\alpha=2$, and the general case is similar.

Note that $\alpha=2$ is the above Theorem \ref{thm:EssNExpan}. Let $X_2=\bigcup_{i\in\N} Y_i\cup S$ and $T_2\colon X_2\to X_2$, $T_2|_S=g\colon S\to S$, $T_2(x_{m,j}^i)=x_{m,j+1}^i$ be defined as in Theorem \ref{thm:EssNExpan}. It suffices to show
$$
\Omega_1(X_2,T_2)=\Omega_1^{(\infty)}(X_2,T_2)=S=\bigcup_{i\in\Z}\{s_i\}\cup\{s_\infty\}.
$$
To show this, we recall the decreasing family of neighborhood systems $V=V_0\supset\dots\supset V_{i}=W_0^i\supset W_1^i\supset \dots\supset W_{n-1}^i\supset V_{i+1}=W_n^i\supset\cdots$. So for any $x\in X_2\setminus S$, by the construction we have $x\in W_{m-1}^i\setminus W_m^i$ for some $m\in [1,n]$ and $i\in\N$. Let $0<\ep <d(x, S)/2$, where $d(x, S)=\min\{d(x,y)\colon y\in S\}$. Then there is $i_0\in\N$ such that $B_\ep(x)\cap V_i=\emptyset$ for all $i\ge i_0$. Hence
$
B_\ep(x)= B_\ep(x) \bigcap (\bigcup_{i=0}^{i_0} \bigcup_{m=1}^n \{x_{m,j}^i\colon j\in \Z\}).
$
On the other hand, $B_\ep(x)\bigcap \{x_{m,j}^i\colon j\in \Z\}$ is a finite set for each $m\in [1,n]$ and $i\in\N$, otherwise by compactness there is another limit point distinct with $s_\infty$, a contradiction with the assumption that $\lim_{j\to +\infty}x_{m,j}^i=\lim_{j\to -\infty}x_{m,j}^i=s_\infty$. It implies that we can choose small $0<\ep_0<\ep$ such that $B_{\ep_0}(x)=\{x\}$ is an open set, and then $x\notin \Omega_1(X_2,T_2)$. Finally the arbitrariness of $x$ yields $X_2\setminus S\cap \Omega_1(X_2,T_2)=\emptyset$.

Clearly $s_\infty\in \Omega_1^{(\infty)}(X_2,T_2)\subset \Omega_1(X_2,T_2)$. Now consider $s_j\ (j\neq\infty)$. As $\bigcap_{i=1}^{+\infty}V_i=S$ and $V_i=\bigcup_{l=-r_i}^{r_i} U_i(s_l)\bigcup U_i(s_\infty) $, there must exist some $i_1\in\N$ such that $-r_i\le j\le r_i$ and then $s_j\in U_i(s_j)$ for all $i\ge i_1$. By the construction, for each $d\in\N$ and $\delta>0$, there are $x\in X_2\setminus S$ and $i_2\in\N$ such that
the prime number $p_i>d$, $x\in  U_i(s_j)\subset B_\delta(s_j)$ and $w_S(x; T_2, V_i)=p_i$ for all $i\ge i_2$. Choose $i\ge \max\{i_1,i_2\}$ and by the definition of winding number, we have some $k\in\Z_+$ and $x'\in X_2\setminus S$ such that $T_2^k x'\in U_i(s_j), \dots,T_2^{dk} x'\in U_i(s_j),T_2^{(d+1)k} x'\in U_i(s_j),\dots,T_2^{p_i\cdot k} x'\in U_i(s_j)$. This implies that
$$
B_\delta(s_j)\cap T_2^{-k} B_\delta(s_j)\cap T_2^{-2k} B_\delta(s_j)\cap\cdots\cap T_2^{-dk} B_\delta(s_j)\neq\emptyset,
$$
and then $s_j\in \Omega_1^{(\infty)}(X_2,T_2)$. As $s_j$ is arbitrary, we have $\Omega_1(X_2,T_2)=\Omega_1^{(\infty)}(X_2,T_2)=S$, which ends the proof.
\end{proof}

In \cite[Remark 6.8]{KLOX15} the authors left a conjecture that there exists a t.d.s. $(X,T)$ such that $\textrm{depth}^{(\infty)}(X,T)=\alpha$. Here we prove that the conjecture is true.

\begin{cor}\label{cor:Van}
Assume $\alpha<\aleph_1$ is an ordinal number. Then there is a countable t.d.s. $(X_\alpha,T_\alpha)$ such that
 $$d(X_\alpha)=\emph{depth}^{(\infty)}(X_\alpha,T_\alpha)=\emph{depth}(X_\alpha,T_\alpha)=\alpha.$$
\end{cor}

\begin{proof}
By Theorem \ref{thm:Van} it remains to show the case that $\alpha$ is a limit ordinal number. Choose a sequence of non-limit ordinals $\alpha_1<\alpha_2<\dots<\alpha$ with $\lim_{i\to +\infty}\alpha_i=\alpha$. For each $i\in\N$, by Theorem \ref{thm:Van} there exists a countable t.d.s. $(X_{\alpha_i},T_{\alpha_i})$ in the plane such that  $d(X_{\alpha_i})=\textrm{depth}^{(\infty)}(X_{\alpha_i},T_{\alpha_i})=\textrm{depth}(X_{\alpha_i},T_{\alpha_i})={\alpha_i}$ and
$\Omega_{\alpha_i}(X_{\alpha_i},T_{\alpha_i})=\Omega_{\alpha_i}^{(\infty)}(X_{\alpha_i},T_{\alpha_i})$. We can further require that the sequence $\{X_{\alpha_i} \}_{i=1}^{+\infty}$ are pairwise disjoint and $\lim_{i\to +\infty}H_d(X_{\alpha_i}, \{x_0\})=0$ for some point $x_0$ in the plane with $x_0\notin X_{\alpha_i}$. Set $X_\alpha=\bigcup_{i=1}^{+\infty}X_{\alpha_i} \cup \{x_0\}$ and $T_\alpha\colon X_\alpha\to X_\alpha$ as $T_\alpha|_{X_{\alpha_i}}=T_{\alpha_i},\ T_\alpha(x_0)=x_0$. It is easy to see that $(X_\alpha, T_\alpha)$ is what we want.
\end{proof}

\subsection{Essential $\aleph_0$-expansiveness} To end this paper, we give an example of essentially $\aleph_0$-expansive homeomorphism.

\begin{exam}\label{exam:CoutExpan}
There exists a t.d.s. $(X,T)$ such that $T$ is essentially $\aleph_0$-expansive.

\smallskip
Let $X=\{0\} \cup \left\{\frac{1}{n}: n \in \mathbb{N} \right\}$ with the subspace topology of the real line $\mathbb{R}$. Define $T: X \rightarrow X$ as
\begin{itemize}
 \item $T(0)=0$ and $T(1)=1$;
 \item $T\left(\frac{1}{2^n}\right)=\frac{1}{2^n+1}, \ldots, T\left(\frac{1}{2^{n+1}-1}\right)=\frac{1}{2^n}$ for each $n \in \mathbb{N}$.
\end{itemize}
Note that for any $\delta>0$ the set $\{y\in X\colon d(T^n 0, T^n y)\le\delta, \forall n\in \Z \}$ is countable infinite. That is what we need.
\end{exam}

%

\section*{Acknowledgments}
We thank Siming Tu for very useful suggestions.

J. Li was supported by NNSF of China (11371339) and R. Zhang
was supported by NNSF of China (11001071,11171320).



\begin{thebibliography}{s2}

\bibitem{Ar14}
A.~Artigue, \emph{Finite sets with fake observable cardinality}, arXiv:1404.0590v1 [math.DS], 2014.

\bibitem{AC14}
A.~Artigue and D.~ Carrasco-Olivera, \emph{A Note on Measure-Expansive Diffeomorphisms}, arXiv:1409.1418v1 [math.DS], 2014.

\bibitem{APV13}
A.~Artigue, M.~J.~Pacifico, J.~L.~Vieitez, \emph{N-expansive homeomorphisms on surfaces}, arXiv:1311.5505v1 [math.DS], 2013.



\bibitem{De89}
R.~Devaney, \emph{Chaotic Dynamical Systems}, Addison-Wesley, Reading MA, 1989.

\bibitem{Ei66}
M.~Eisenberg, \emph{Expansive transformation semigroups of endomorphisms}, Fund. Math. \textbf{59}(1966), 313--321.


\bibitem{HW48}
W.~Hurewicz and H.~Wallman, \emph{Dimension theory}, Princeton Univ. Press, Princeton, N.J., 1948.


\bibitem{Ka93}
H.~Kato, \emph{Continuum-wise expansive homeomorphisms}, Canad. J. Math \textbf{45}(1993), no. 3, 576--598.

\bibitem{KP99}
H.~Kato and J.~J.~Park, \emph{Expansive homeomorphisms of countable compacta}, Topology and its Applications \textbf{95}(1999), no. 3, 207--216.

\bibitem{KR69}
H.~B.~Keynes and J.~B.~Robertson, \emph{Generators for topological entropy and expansiveness}, Theory of Computing Systems \textbf{3}(1969), no.1, 51--59.

\bibitem{Ku68}
K.~Kuratowski, \emph{Topology}, vol. I, 1968.

\bibitem{KLOX15}
D.~Kwietniak, J.~Li, P.~Oprocha and X.~D.~Ye, \emph{Multi-recurrence and van der Waerden systems}, preprint (2015).

\bibitem{MS10}
J.~H.~Mai and W.~H.~Sun, \emph{Positively expansive homeomorphism on metric space}, Acta Math. Hungar. \textbf{126}(2010), no. 4, 366--368.

\bibitem{Ma79}
R.~Ma\~n\'e, \emph{Expansive homeomorphisms and topological dimension}, Transactions of the American Mathematical Society \textbf{252}(1979), 313--319.

\bibitem{Mo12}
C.~Morales, \emph{A generalization of expansivity}, Discrete and Continuous Dynamical Systems \textbf{32}(2012), no. 1, 293--301.

\bibitem{Ne78}
D.~A.~Neumann, \emph{Central sequences in dynamical systems}, Amer. J. Math. \textbf{100}(1978), 1--18.

\bibitem{Re65}
W.~Reddy, \emph{The existence of expansive homeomorphisms on manifolds}, Duke Math. J. \textbf{32}(1965), 627--632.


\bibitem{Ut50}
W.~R.~Utz, \emph{Unstable homeomorphisms}, Proc. Amer. Math. Soc. \textbf{1}
(1950), 769--774.

\bibitem{Wa00}
P.~Walters, \emph{An introduction to ergodic theory}, Grad.Texts in Math., Springer, 2000.

\end{thebibliography}
\end{document}